\documentclass{amsart}
\usepackage[foot]{amsaddr}
\usepackage[a4paper]{geometry}
\usepackage{amsmath,amssymb}
\usepackage{graphicx}
\usepackage{mathtools}
\usepackage{bm}
\usepackage[hidelinks]{hyperref}
\usepackage{cleveref}
\usepackage{enumitem}
\usepackage{tikz-cd}
\usepackage{biblatex}
\usepackage{dutchcal}
\usepackage{IMjournal}

\numberwithin{equation}{section}

\newtheorem{theorem}{Theorem}[section]
\newtheorem{lemma}[theorem]{Lemma}
\newtheorem{definition}[theorem]{Definition}

\DeclareMathOperator{\divergence}{div}
\DeclareMathOperator{\gradient}{grad}
\DeclareMathOperator{\curl}{curl}
\DeclareMathOperator{\Ker}{Ker}

\newcommand{\norm}[1]{\left\lVert#1\right\rVert}

\bibliography{references}

\begin{document}
\title{Fitted norm preconditioners for the Hodge Laplacian in mixed form}

\author{\| Wietse M.  |Boon|, Bergen,
        \| Johannes |Kraus|, Essen,
        \| Tomáš |Luber|, Ostrava,
        \| Maria |Lymbery|, Essen} 

\rec{July 31, 2025}

\abstract 
    We use the practical framework for abstract perturbed saddle point problems recently introduced by Hong et al.~to analyze the mixed formulation of the Hodge Laplace problem. We compose two parameter-dependent norms in which the uniform continuity and stability of the problem follow. This not only guarantees the well-posedness of the corresponding variational formulation on the continuous level, but also of related compatible discrete models.
    
    We further simplify the obtained norms and, in both cases, arrive at the same norm-equivalent preconditioner that is easily implementable. The efficiency and uniformity of the preconditioner are demonstrated numerically by the fast convergence and uniformly bounded number of preconditioned MINRES iterations required to solve various instances of Hodge Laplace problems in two and three space dimensions.
\endabstract

\keywords
  Preconditioning, Hodge Laplacian, perturbed saddle point, Hilbert complex
\endkeywords

\subjclass
65N22, 
65F08, 
35J05, 
58J10, 
65N30,  
58A14 
\endsubjclass


\section{Introduction}
\label{sec:introduction}

Effective preconditioners are essential for solving discretized systems of partial differential equations. 
It is of particular importance that the preconditioner is robust with respect to physical and discretization parameters, so that the convergence of the iterative solver can be guaranteed in a wide range of applications. 

In multi-physics problems, the systems of equations typically have a block structure where the blocks model the different physical systems and the coupling relations. With this perspective, coupled systems can often be identified as (perturbed) saddle point problems \cite{boffi2013mixed}. Examples include flow in porous media in which the fluid flux is coupled to the fluid pressure, but also Biot poroelasticity in which the solid mechanics is coupled to the fluid flow \cite{Cheng2016}, and the Maxwell equations \cite{Monk2003-vd}. 

A recent paper \cite{hong2023new} establishes a framework that exploits this block structure to provide a robust block diagonal preconditioners for perturbed saddle point systems. A main advantage of this approach is that the norms are provided by the theory. In contrast, the preconditioners proposed in \cite{boon2021robust} require the practitioner to derive appropriate norms and verify an inf-sup condition from \cite{Braess1996stability}. 
Since the framework~\cite{hong2023new} provides these norms, it can directly be used to construct norm-equivalent preconditioners~\cite{mardal2011preconditioning}. 
It is thereby applicable on the continuous level, as demonstrated in~\cite{hong2023new} for various mixed variational formulations of generalized Poisson, Stokes, vector Laplace and Biot's equations but also to discrete models, as shown, e.g., for a conservative discretization of the Biot-Brinkman problem in~\cite{HoKrKuLyMaRo2021}, or a hybridized hybrid-mixed-method for the multiple network poroelasticity problem in~\cite{KrLeLyOsSc2022}. 
The objective of this paper is to explore the application of this framework to the case of the Hodge Laplacian in mixed form.

The (weighted) abstract Hodge Laplace problem is posed as: Given $f$, find $u$ such that 
\begin{equation}
    (\alpha^{-1} d_{k-1} d^*_{k-1} + d^*_{k} d_{k}) u = f,
\end{equation}
where $d_{k - 1}$ and $d_k$ are differential operators that satisfy $d_k d_{k - 1} = 0$, and $d_{k - 1}^*$ and $d_k^*$ denote their respective adjoints. 
A prominent example is the Laplace problem in which $d_{k - 1} = \divergence$ and $d_k = 0$. This problem is ubiquitous in physics simulations \cite{evans2022partial}, with applications ranging from Darcy flow to Fourier heat conduction, in which $\alpha$ denotes the permeability, respectively heat conductivity, of the medium. On the other hand, the vector Laplacian with $d_{k} = \divergence$ and $d_{k - 1} = \curl$ plays a role in Navier-Stokes equations and the Maxwell equations. More recently, the Hodge Laplace structure has been recognized in the modeling of Cosserat elasticity \cite{vcap2024bgg} and in the equations governing flow in fractured porous media \cite{boon2021functional}.

In this work, we consider the mixed formulation of the Hodge Laplacian problem, that treats $d^*_k$ weakly, in contrast with the usual primal approach \cite{Savostianov2024}. 
We then use the norm fitting framework to show that the problem is well posed and construct robust and efficient block diagonal preconditioners for Krylov subspace solvers.

The novelty of the presented approach lies in the unified analysis of the Hodge Laplacian on a Hilbert complex.
Our analysis will follow the line of operator preconditioning \cite{mardal2011preconditioning}.
We will analyze the system and derive all results on the functional level and use a structure preserving simplicial discretization coming from Finite Element Exterior Calculus \cite{arnold2006finite} to transfer them to the discrete level, leading to well posed linear systems and robust matrix preconditioners.
Discretization on cubical meshes can be found in \cite{Lee2022} and on polytopal meshes in \cite{bonaldi2025exterior,di2023arbitrary}

The starting point of the analysis is the abstract exact Hilbert complex \cite{arnold2018finite, Hiptmair2023}.
We will make use of its properties, namely the Poincaré inequality and Hodge decomposition, and the key ideas from \cite{hong2023new} to propose a preconditioner scheme that can be applied for every system that fits the Hilbert complex setup. 
The results presented here are therefore applicable to any complex that satisfies these properties, like de Rham complex on Riemannian manifolds \cite{Bruning1992} or complexes arising in elasticity \cite{Arnold, Angoshtari2016}.

The remainder of this article is organized as follows. \Cref{sec:hilbert_complexes} introduces the Hodge Laplace problem in the context of Hilbert complexes. The analysis framework of \cite{hong2023new} for perturbed saddle point systems is summarized in \Cref{sec:perturbed_saddle_point_problems}. We use this framework to analyze the Hodge Laplace problem in \Cref{sec:analysis_HL}. Finally, numerical experiments are presented in \Cref{sec:numerical_results}, showcasing the robustness of the proposed preconditioners and conclusions are presented in \Cref{sec:conclusion}.

\section{Hilbert complexes and the Hodge Laplace problem}
\label{sec:hilbert_complexes}

The Hodge Laplacian offers a generalization of the usual scalar Laplacian to vector-valued functions on manifolds.
We will use the language of Hilbert complexes, see \cite{arnold2018finite}, which provides a unifying framework for treatment and generalization of the common differential operators $\gradient, \divergence, \curl$ and their compositions, like the Laplacian $\Delta$. Moreover, the theory provides the machinery to construct stable, conforming discretizations through Finite Element Exterior Calculus \cite{arnold2006finite}.

Let $\Omega \subset \mathbb{R}^n$ be a bounded Lipschitz domain. 
For $n = 3$, the relevant differential operators and Hilbert spaces can be written in the following diagram, known as the de Rham complex
\begin{equation} \label{eq:example_sequence}
\begin{tikzcd}
    0 \rar  
    & H^1(\Omega) \arrow[r, "\gradient"] 
    & \mathbf{H}(\curl, \Omega) \arrow[r, "\curl"] 
    & \mathbf{H}(\divergence, \Omega) \arrow[r, "\divergence"] 
    & L^2(\Omega) \rar
    & 0.
\end{tikzcd}
\end{equation}

Here, every arrow is a bounded linear map with the arrows going from and to $0$ representing trivial maps. The calculus identities $\curl\gradient = 0, \divergence\curl = 0$ imply that the image of each differential operator lies in the kernel of the next differential operator. Let $H \Lambda^\bullet$ and $d_\bullet$ serve as compact notation for the spaces and differential operators, then \eqref{eq:example_sequence} can be structurally represented as
\begin{equation} \label{eq:deRham_as_hilbert_complex}
    \begin{tikzcd}
        0 \rar
        & H\Lambda^0 \arrow[r, "d_0"]
        & H\Lambda^1 \arrow[r, "d_1"]
        & H\Lambda^2 \arrow[r, "d_2"]
        & H\Lambda^3 \rar
        & 0.
    \end{tikzcd}
\end{equation}

We refer to $(H \Lambda^\bullet, d_\bullet)$ as a \emph{Hilbert complex} if each $H \Lambda^k$ is a Hilbert space and the differentials $d_k : H \Lambda^k \to H \Lambda^{k + 1}$ satisfy $d_{k + 1} d_k = 0$. We omit the subscript on $d$ when clear from context. We call a Hilbert complex \emph{exact} if for each $q \in H \Lambda^k$ with $d q = 0$ there exists a $v \in H \Lambda^{k - 1}$ such that $dv = q$. 
In this work, we only consider exact complexes. In the context of the de Rham complex \eqref{eq:example_sequence}, it suffices to assume that $\Omega$ is contractible.


With these definitions in place, we can pose the mixed formulation of the Hodge Laplace problem. 
Given the functionals $g \in (H\Lambda^{k - 1})'$ and $f \in (H\Lambda^k)'$, find $(p, u) \in H \Lambda^{k -1} \times H \Lambda^k$ such that
\begin{subequations} \label{eqs: Hodge Laplace problem original}
	\begin{align}
	(\alpha p, q)_\Omega - (u, dq)_\Omega &= \langle g, q \rangle, &
	\forall q &\in H \Lambda^{k - 1}, \\
	(dp, v)_\Omega + (du, dv)_\Omega &= \langle f, v \rangle, &
	\forall v &\in H \Lambda^k.
\end{align}
\end{subequations}
Here, $(\cdot, \cdot)_\Omega$ is the $L^2$ inner product on $\Omega$ and $\langle \cdot, \cdot \rangle$ denotes a duality pairing. $\alpha$ is a positive constant that often represents a material parameter. 

Let the norm on $H \Lambda^k$ be given by $\norm{v}^2_k = \norm{v}_\Omega^2 + \norm{d_k v}_\Omega^2$. Let
$d^*_k$ denote the adjoint to $d_k$, i.e. it satisfies the relation
\begin{align} \label{eq:dstar_def}
    (d_k^* z, v)_\Omega &= 
    (z, d_k v)_\Omega, &
    \forall (z, v) \in H\Lambda^{k + 1} \times H\Lambda^k
\end{align}
Note that, in general, it is not guaranteed that $d_k^* z \in H \Lambda^k$ for all $z$. However, it is guaranteed if $H \Lambda^\bullet$ is a discrete subcomplex of \eqref{eq:example_sequence}, which is the case considered in this work.

To analyze the mixed formulation, we will use the following standard results for an exact complex, see \cite{arnold2018finite, arnold2006finite}. 
\begin{lemma}[Poincaré inequality] \label{lemma:Poincare} 
    There exist constants $c_k^P$ such that
    \begin{equation}
        \| u \|_\Omega^2 \leq c_k^P \| d u \|_\Omega^2 \quad \forall u \in H \Lambda^{k} \text{ s.t. } d_{k-1}^* u = 0.
    \end{equation}
\end{lemma}

\begin{lemma}[Hodge decomposition] \label{lemma:Hodge}
    Every $v \in H\Lambda^k$ can be uniquely decomposed as
    \begin{equation}
        v = d_{k-1} w + d^*_k z
    \end{equation}
    where $w \in H \Lambda^{k-1}$ with $d_{k - 2}^*w = 0$, $z \in H\Lambda^{k+1}$ with $d_{k + 1} z = 0$, and the summands are orthogonal in the $L^2$ inner product.
\end{lemma}

\section{The fitted norm framework for perturbed saddle point problems}
\label{sec:perturbed_saddle_point_problems}

\subsection{Perturbed saddle point problems}

Let $V$ and $Q$ be Hilbert spaces equipped with norms
$\Vert \cdot \Vert_V$ and $\Vert \cdot \Vert_Q$ that are induced by scalar products 
$(\cdot,\cdot)_V$ and $(\cdot,\cdot)_Q$, respectively. 
We introduce the product space $Y:=V\times Q$ and equip it  
with the norm 
\begin{equation}\label{y_norm}
\Vert y \Vert_Y^2 = (y,y)_Y = (v,v)_V + (q,q)_Q = \Vert v\Vert_V^2 + \Vert q\Vert_Q^2  \qquad \forall y = 
(v;q):= \begin{pmatrix}v\\ q\end{pmatrix}\in Y. 
\end{equation}

On $Y\times Y$ we consider the abstract bilinear form
\begin{equation}\label{bilinear_form}
\mathcal{a}((u;p),(v;q)):=a(u,v)+b(v,p)+b(u,q)-c(p,q)
\end{equation}
for some symmetric positive semidefinite (SPSD)
bilinear forms 
$a(\cdot,\cdot)$ on 
$V\times V$ and SPSD $c(\cdot,\cdot)$ on $Q \times Q$
and a bilinear form 
$b(\cdot,\cdot)$ on $V \times Q$.

Under the assumption that the continuity bounds
\begin{subequations}
\begin{align}
    a(u,v) & \le \bar{C}_a \Vert u\Vert_V \Vert v\Vert_V  & \forall u,v &\in V, 
    \label{a_cont}\\
    b(v,q) & \le \bar{C}_b \Vert v\Vert_V \Vert q\Vert_Q  & \forall v &\in V,\forall q\in Q, 
    \label{b_cont}\\
    c(p,q) & \le \bar{C}_c \Vert p\Vert_Q \Vert q\Vert_Q  & \forall p,q &\in Q
\label{c_cont}
\end{align} 
\end{subequations}
hold, we define the bounded linear operators: 
\begin{subequations}\label{eqs:operatorsABC}
\begin{align}
    A&: V\rightarrow V' & \langle Au,v\rangle_{V'\times V} &\coloneqq a(u,v), & \forall u,v &\in V, \\
    C&: Q\rightarrow Q' & \langle Cp,q\rangle_{Q'\times Q} &\coloneqq c(p,q), & \forall p,q &\in Q, \\
    B&: V\rightarrow Q' & \langle Bv,q\rangle_{Q'\times Q}  &\coloneqq b(v,q), & \forall (v,q) &\in V \times Q, \\ 
    B^T&: Q\rightarrow V' & \langle B^T q, v\rangle_{V'\times V} &\coloneqq b(v,q), & \forall (v,q) &\in V \times Q,
\end{align}
\end{subequations}
where, $V'$ and $Q'$ are the dual spaces of $V$ and $Q$, 
while $\langle \cdot, \cdot \rangle $ denote the corresponding duality pairings.  

The associated perturbed saddle-point problem is given by: find $(u, p) \in V \times Q$ such that
\begin{equation}\label{PSPP}
\mathcal{a}((u;p),(v;q)) = \mathcal{F}((v;q))\qquad \forall (v, q) \in Y, 
\end{equation}
where $\mathcal{F}\in Y'$, i.e., $\mathcal{F}: Y\rightarrow \mathbb{R}: \mathcal{F}(y)=\langle \mathcal{F},y \rangle_{Y'\times Y}$ 
for all $y \in Y$.
Using the definitions $x=(u;p)$ and $y=(v;q)$, \eqref{PSPP} can be rewritten as 
$$\mathcal{a}(x,y) = \mathcal{F}(y)\qquad \forall y \in Y,$$
or, in operator form 
\begin{equation}\label{PSPP_o}
\mathcal{A} x = \mathcal{F}
\end{equation}
where 
\begin{align}
\mathcal{A}&: Y\rightarrow Y' & \langle \mathcal{A}x,y\rangle_{Y'\times Y} &\coloneqq \mathcal{a}(x,y), & \forall x,y &\in Y.
\end{align}

The operator $\mathcal{A}$ has the block form  
\begin{equation}\label{A_block_form}
\mathcal{A} = \begin{pmatrix}
A & B^T \\
B & -C
\end{pmatrix}.
\end{equation}

\begin{theorem}[Babu\v{s}ka~\cite{Babuska1971error}]\label{thm:1} 
Consider the bounded linear functional $\mathcal{F}\in Y'$. 
The saddle-point problem~\eqref{PSPP} is well-posed if 
the conditions 
\begin{equation}\label{boundedness_A}
    \mathcal{a}(x, y) \le \bar{C} \Vert x\Vert_{Y} \Vert y\Vert_{Y} \qquad \forall x, y \in Y,
\end{equation}
\begin{equation}\label{inf_sup_A}
\inf_{x \in Y} \sup_{y \in Y} \frac{\langle \mathcal{A} x,y \rangle}{\Vert x\Vert_{Y} \Vert y\Vert_{Y}} \ge \underline{\alpha}>0
\end{equation}
can be fulfilled with some positive constants $\bar{C}$ and $\alpha$.
The stability estimate 
\begin{equation*}
\Vert x\Vert_{Y} \le \frac{1}{ \underline{\alpha}} \sup_{y \in Y}\frac{\mathcal{F}(y)}{\Vert y\Vert_{Y}}
=: \frac{1}{ \underline{\alpha}}\Vert \mathcal{F}\Vert_{Y'}
\end{equation*}
then holds true for the solution $x$.
\end{theorem}

\subsection{Fitted norms}
\label{sub:fitted_norms}

In this paper we make use of the practical framework for the parameter-robust stability analysis of perturbed saddle-point problems
introduced in~\cite{hong2023new}. This framework provides a norm fitting technique on an abstract level, which applies to the variational
formulation of various multiphysics models. 

Hereby, the form $c(\cdot,\cdot)$ causing the perturbation of problem~\eqref{PSPP}, contributes to the norm $\Vert \cdot \Vert_Q$
on the space $Q$ according to the definition
\begin{equation} \label{Q_norm}
\Vert q\Vert_Q^2 :=\vert q\vert_Q^2+c(q,q) =: \vert q\vert_Q^2+\vert q\vert_c^2=: \langle \bar{Q}q,q \rangle_{Q'\times Q}.
\end{equation}
Note that the requirement that $\Vert \cdot \Vert_Q$ is a full norm induced by an inner product under which $Q$ is a
Hilbert space already implies that the seminorm $\vert \cdot \vert_Q$ is also induced by an SPSD bilinear form
$d(\cdot,\cdot):Q\times Q \rightarrow \mathbb{R}$, i.e., $\vert q\vert_Q^2=d(q,q)$.
Thus the bilinear form $c(p,q)+d(p,q)$ has to be symmetric positive definite (SPD) on $Q$, defining a linear operator
$\bar{Q}: Q\rightarrow Q'$ by $\langle \bar{Q}p,q\rangle := c(p,q)+d(p,q)$.

Next, the norm $\Vert \cdot \Vert_V$ on the space $V$ is split into two seminorms according to the definition
\begin{align}\label{V_norm}
\Vert v \Vert^2_{V} &:= \vert v\vert^2_{V} +\vert v \vert^2_{b}
\end{align}
where $\vert \cdot \vert_V$ becomes a norm on ${\rm Ker}(B)$ satisfying  
$$
\vert v\vert_V^2 \eqsim  a(v,v) \qquad \forall v \in {\rm Ker}(B).
$$ 
Moreover, and this is the key idea of the construction, $\vert \cdot\vert_b$ is defined by 
\begin{align}\label{seminorm_b}
\vert v \vert^2_{b}:= \langle Bv, \bar{Q}^{-1}Bv \rangle_{Q' \times Q} = 
\Vert Bv \Vert_{Q'}^2. 
\end{align}
relating the two norms $\Vert \cdot \Vert_Q$ and $\Vert \cdot \Vert_V$ to each other. 

Then $\bar{Q}^{-1}: Q' \rightarrow Q$ is an isometric isomorphism because $\bar{Q}$ is the Riesz isomorphism, i.e.,   
\begin{align*}
\Vert \bar{Q}^{-1}Bv\Vert_Q^2 =& (\bar{Q}^{-1}Bv,\bar{Q}^{-1}Bv)_Q 
= \langle \bar{Q}\bar{Q}^{-1}Bv,\bar{Q}^{-1}Bv \rangle_{Q'\times Q} \\
=& \langle Bv, \bar{Q}^{-1}Bv\rangle_{Q'\times Q} =\Vert Bv\Vert_{Q'}^2.
\end{align*}

Note that both $\vert \cdot \vert_V$ and $\vert \cdot \vert_b$ can be seminorms, one of them could even be the
trivial seminorm (zero), provided their sum is a norm. The same applies to the seminorms $\vert \cdot \vert_Q$
and $\vert \cdot \vert_c$ which have to add up to a full norm only.
The splitting~\eqref{V_norm} gives rise to a Schur complement type operator, corresponding to a regularized bilinear
form resulting from $\mathcal{a}((\cdot;\cdot),(\cdot;\cdot))$ by replacing $c(\cdot,\cdot)$ with $(\cdot,\cdot)_Q$.

The main theoretical result in~\cite{hong2023new} uses the following notion of {\em{fitted norms}}.

{\begin{definition}[Hong et al.~\cite{hong2023new}]\label{A1}
For two Hilbert spaces $V$ and $Q$, a norm $\Vert \cdot\Vert_V$ on $V$ and a norm $\Vert \cdot\Vert_Q$ on $Q$ 
are called fitted if they satisfy the splittings~\eqref{Q_norm} and~\eqref{V_norm}, respectively, where $\vert \cdot\vert_Q$ 
is a seminorm on~$Q$ and $\vert \cdot \vert_V$ and $\vert \cdot\vert_b$ are seminorms on $V,$ 
the latter defined according to~\eqref{seminorm_b}.
\end{definition}}


The following theorem in~\cite{hong2023new} can be exploited for constructing a parameter-robust
preconditioner for the Hodge Laplacian in mixed form.
\begin{theorem}[Hong et al.~\cite{hong2023new}]\label{thm:2}
Let $\Vert \cdot\Vert_V$ and $\Vert \cdot \Vert_Q$ be {fitted} norms according to Definition~\ref{A1}, 
which immediately implies the continuity of $b(\cdot,\cdot)$ and $c(\cdot,\cdot)$
in these norms with $\bar{C}_b=1$ and
$\bar{C}_c=1$, cf.~\eqref{b_cont}--\eqref{c_cont}.
Consider the bilinear form $\mathcal{a}((\cdot;\cdot),(\cdot;\cdot))$ 
defined in~\eqref{bilinear_form} where $a(\cdot,\cdot)$ is continuous,
i.e., \eqref{a_cont} holds, and $a(\cdot,\cdot)$ and 
$c(\cdot,\cdot)$ are symmetric positive semidefinite. 
Assume, further, that $a(\cdot,\cdot)$ satisfies the coercivity estimate
\begin{equation}\label{coerc_a}
a(v,v)\ge \underline{C}_a \vert v\vert_V^2, \qquad \forall v\in V,
\end{equation}
and that there exists a constant 
$\underline{\beta}>0$ such that 
\begin{equation}\label{inf_sup_b}
\sup_{\substack{v\in V\\ v\neq 0}} 
\frac{b(v,q)}{\Vert v\Vert_V}\ge \underline{\beta}  \vert q\vert_Q, \qquad \forall q\in Q.
\end{equation} 
Then the bilinear form $\mathcal{a}((\cdot;\cdot),(\cdot;\cdot))$ is continuous and inf-sup stable under the combined norm 
$\Vert \cdot \Vert_Y$ defined in~\eqref{y_norm}, i.e., the conditions~\eqref{boundedness_A} and \eqref{inf_sup_A} hold.
\end{theorem}

Note that the norm fitting can also be performed by first fixing the full norm on $V$ (instead of the full norm on $Q$ as
described above), resulting in the splitting
\begin{subequations} \label{eqs:flipped_norms}
    \begin{align}
    \Vert v \Vert_V^2 &\coloneqq  \vert v \vert_V^2 + a(v,v)=:\langle \bar{V}v,v \rangle_{V'\times V},\\
    \Vert q \Vert_Q^2 &\coloneqq  \vert q \vert_Q^2 + \vert q\vert_b^2,
\end{align}
\end{subequations}
where $\bar{V}:V\rightarrow V'$ is an invertible linear operator. In this case, the roles of $a$ and $c$ are reversed, as well as the roles of $v$ and $q$, i.e. $\vert q\vert_Q^2$ is equivalent to $c(q,q)$
and $\vert q\vert_b^2 \coloneqq \langle B^T q, \bar{V}^{-1}B^Tq \rangle_{V'\times V}$.

\section{Analysis of the Hodge Laplace problem using fitted norms}
\label{sec:analysis_HL}

The goal of this work is to use the fitted norm approach presented in~\cite{hong2023new} to construct a norm-equivalent preconditioner $\mathcal{P}$ for the operator~\eqref{A_block_form}
associated with problem~\eqref{eqs: Hodge Laplace problem original}.
In order to guarantee that the condition number of the preconditioned operator $\mathcal{P} \mathcal{A}$ then neither depends on the parameter $\alpha$ in problem~\eqref{eqs: Hodge Laplace problem original}
nor on any discretization parameters in a compatible discrete version of it, we first study the Hodge Laplacian in mixed form both on the continuous and discrete level.

\subsection{An inf-sup condition}

We consider the linear operator~$\mathcal{A}: Y \rightarrow Y'$ of the form~\eqref{A_block_form} with $Y:=V\times Q$ where $V \coloneqq H \Lambda^k$ and $Q \coloneqq H \Lambda^{k - 1}$
associated with the system~\eqref{eqs: Hodge Laplace problem}, which, for convenience, we rewrite in the equivalent form
\begin{subequations} \label{eqs: Hodge Laplace problem}
	\begin{align}
	(du, dv)_\Omega + (dp, v)_\Omega &= \langle f, v \rangle, &
	\forall v &\in H \Lambda^k, \\
	 (u, dq)_\Omega - (\alpha p, q)_\Omega&= - \langle g, q \rangle, &
	\forall q &\in H \Lambda^{k - 1}.	
\end{align}
\end{subequations}
The bounded linear operators $A, B$ and $C$ in~\eqref{eqs:operatorsABC} can then be identified as
\begin{subequations}\label{eqs:ABC_HodgeLaplacian}
\begin{align}
	\langle A u, v \rangle &\coloneqq (du, dv)_\Omega 
    , \label{eq:A_HodgeLaplacian} \\
	\langle B u, q \rangle &\coloneqq (u, dq)_\Omega 
    , \label{eq:B_HodgeLaplacian} \\
	\langle C p, q \rangle &\coloneqq (\alpha p, q)_\Omega 
    . \label{eq:C_HodgeLaplacian}
\end{align}
\end{subequations}
Following~\cite[Example 3.5]{hong2023new} we define the following seminorms on the Hilbert spaces $V \coloneqq H \Lambda^k$ and $Q \coloneqq H \Lambda^{k - 1}$:
\begin{subequations}\label{eqs:VQ_seminorms}
\begin{align}
	|v|_V^2 &\coloneqq \| dv \|_\Omega^2, \label{eq:V_seminorm} \\
	|q|_Q^2 &\coloneqq (1 + \alpha) \| dq \|_\Omega^2, \label{eq:Q_seminorm}
\end{align}
\end{subequations}
and recall that according to the general framework $|v|_b^2 = \langle B v, \bar Q^{-1} B v \rangle$ with $\bar Q$ defined in~\eqref{Q_norm}.
The fitted norms introduced in \Cref{sub:fitted_norms} now become
\begin{subequations}\label{eqs:VQ_norms}
\begin{align}
	\| q \|_Q^2 
	&= |q|_Q^2 + \langle C q, q \rangle
	= \alpha \| q \|_\Omega^2 + (1 + \alpha) \| dq \|_\Omega^2, \label{eq:Q_norm} \\
	\| v \|_V^2 
	&= |v|_V^2 + |v|_b^2
	= \| dv \|_\Omega^2 + ((\alpha I + (1 + \alpha) d^*d)^{-1} d^* v, d^* v)_\Omega .\label{eq:V_norm}	
\end{align}
\end{subequations}
%
We now proceed to verify the assumptions of Theorem~\ref{thm:2}.

First, in view of the definitions~\eqref{eq:A_HodgeLaplacian},~\eqref{eq:V_seminorm} and ~\eqref{eq:V_norm}, we note that both continuity of $a(\cdot,\cdot)$
in $V$-norm as well as coercivity of $a(\cdot,\cdot)$ in~$V$-seminorm are obvious; the corresponding estimates~\eqref{a_cont} and \eqref{coerc_a} both are
satisfied trivially, i.e., for $\bar{C}_a=1$ and $\underline{C}_a=1$.

\begin{lemma} \label{lemma:0}
	$| \cdot |_V$ is a norm on ${\rm Ker}(B)$.
\end{lemma}
\begin{proof}
	If $v \in {\rm Ker}(B)$, then $d^*v = 0$. Therefore, the Poincaré inequality (\Cref{lemma:Poincare}) gives $\| v \|_\Omega \lesssim \| dv \|_\Omega = |v|_V$ from which we conclude that it is a norm.
\end{proof}

Next, we prove the inf-sup condition~\eqref{inf_sup_b} for the operator $B$ defined in~\eqref{eq:B_HodgeLaplacian}.

\begin{lemma}\label{lemma:1}
The bilinear form $b(\cdot,\cdot)$ defining the linear operator $B$ in~\eqref{eq:B_HodgeLaplacian} is uniformly inf-sup stable in $\alpha$
for the fitted norms given by~\eqref{eqs:VQ_norms}. Further, the constant $\underline{\beta}$ in condition~\eqref{inf_sup_b} satisfies the
estimate $\underline{\beta} \ge 1$ for this choice of (semi)norms.
\end{lemma}

\begin{proof}
According to~\eqref{eq:B_HodgeLaplacian} we have $B=d^*: V \rightarrow Q'$.
For given $q\in Q$, choose $v_0=dq \in V$ to obtain
\begin{align} \label{eq: calculation}
\Vert v_0 \Vert_V^2&= \| d dq \|_\Omega^2 + \left((\alpha I+ (1 + \alpha) d^*d)^{-1}   d^*dq, d^*dq \right)_\Omega \nonumber \\
&=\left((\alpha I+ (1 + \alpha) d^*d)^{-1}   d^*dq, d^*dq \right)_\Omega \nonumber \\
&\le \left(q, (1 + \alpha)^{-1} d^*dq \right)_\Omega \nonumber \\
&= (1 + \alpha)^{-1} \| dq \|_\Omega ^2
\end{align} 
and therefore
\begin{align*}
\sup_{v \in V} \frac{b(v,q)}{\Vert v \Vert_V} &\ge \frac{b(v_0,q)}{\Vert v_0\Vert_V} = \frac{(dq,  dq)_\Omega}{\Vert v_0\Vert_V}  \\
&\ge  \frac{\| dq \|_\Omega^2}{(1 + \alpha)^{-\frac{1}{2}} \| dq \|_\Omega}= \vert q \vert_Q, 
\end{align*}
showing that~\eqref{inf_sup_b} holds true with $\underline{\beta} \ge 1$.
\end{proof}

\subsection{A norm-equivalent preconditioner}

Lemma~\ref{lemma:1} together with the other observations in the previous subsection show that the assumptions in Theorem~\ref{thm:2}
are satified for the fitted norms given by~\eqref{eqs:VQ_norms}
and they therefore represent a robust block diagonal preconditioner.

The $V$-norm given by~\eqref{eq:V_norm}, however, is not convenient in practice. The following lemma provides
an equivalent norm that greatly simplifies the preconditioner.

\begin{lemma}\label{lemma:2}
The norm $ (1 + \alpha)^{-1} \| v \|_\Omega^2 + \| dv \|_\Omega^2$ is equivalent to the $V$-norm defined in~\eqref{eq:V_norm}, that is,
\begin{align}\label{eq:equivalent_V_norms}
\| v \|_V^2 \eqsim (1 + \alpha)^{-1} \| v \|_\Omega^2 + \| dv \|_\Omega^2, & \qquad \forall v \in V = H \Lambda^k.
\end{align}
\end{lemma}

\begin{proof}
In the proof, to avoid any ambiguity, we use the more explicit notation $d_k : H \Lambda^k \to H \Lambda^{k + 1}$
for the differential $d_k$ satisfying $d_{k + 1} d_k = 0$. 

Let $v \in H \Lambda^{k}$ be arbitrary and let its Hodge decomposition, cf.~\Cref{lemma:Hodge}, be given by
\begin{align}\label{eq:Helmholtz_1}
	v = d_{k-1} w + d^*_{k} z
\end{align}
with $w \in H \Lambda^{k-1}$ and $z \in H \Lambda^{k+1}$.

We first show the $\lesssim$ direction:
\begin{align}
\Vert v \Vert_V^2 &= \vert v \vert_V^2+\langle 
B v, \bar Q ^{-1} B v \rangle_{Q'\times Q} \nonumber\\
&= \| d_k v \|_\Omega^2 + \left((\alpha I+(1 + \alpha) d_{k-1}^* d_{k-1})^{-1} d_{k-1}^* v, d_{k-1}^* v \right)_\Omega
\nonumber\\
&= \| d_k v \|_\Omega^2 + \left((\alpha I+(1 + \alpha) d_{k-1}^* d_{k-1})^{-1} d_{k-1}^* d_{k - 1}w, d_{k-1}^* d_{k - 1}w \right)_\Omega
\nonumber\\
&\le \| d_k v \|_\Omega^2 + (1 + \alpha)^{-1} \| d_{k - 1} w \|_\Omega^2
\nonumber\\
&\le \| d_k v \|_\Omega^2 + (1 + \alpha)^{-1} \| v \|_\Omega^2
\end{align}
where the first inequality is due to \eqref{eq: calculation} and the final inequality follows from the orthogonality of the Hodge decomposition. 
For the norm equivalence~\eqref{eq:equivalent_V_norms}, it remains to show the converse
direction $\gtrsim$, i.e. that
\begin{align}\label{eq:V_norms_bound_1}
(d_k v, d_k v )_\Omega + \left((1 + \alpha)^{-1} v, v\right)_\Omega \leq c_1 \Vert v \Vert_V^2
\end{align}
for some constant $c_1 > 0$ independent of $\alpha$. We prove this by bounding the two components of the Hodge decomposition separately.

For the first component, we recall that the Poincar\'e inequality on $H \Lambda^{k-1}$ (\Cref{lemma:Poincare}) reads
\begin{align}\label{eq:Poincare_1}
(q, q)_\Omega \leq c_{k - 1}^P (d_{k-1} q, d_{k-1} q)_\Omega = c_{k - 1}^P (d_{k-1}^* d_{k-1} q, q)_\Omega, & \quad \forall q \in H \Lambda^{k-1}
\ \mbox{s.t.} \  d_{k-2}^* q = 0,
\end{align}
where $c_{k - 1}^P$ denotes the Poincar\'e constant (respectively its square) on $H \Lambda^{k - 1}$.
Hence,
\begin{align*}
((\alpha I + (1 + \alpha) d_{k-1}^* d_{k-1}) q, q )_\Omega & \leq ((\alpha c_{k - 1}^P+ (1 + \alpha)) d_{k-1}^* d_{k-1} q, q )_\Omega \\
& \le (c_{k - 1}^P+1)(1 + \alpha) (d_{k-1}^* d_{k-1} q, q )_\Omega,
\end{align*}
or, equivalently,
\begin{align}\label{eq:V_norms_bound_2}
((\alpha I + (1 + \alpha) d_{k-1}^* d_{k-1})^{-1} q, q )_\Omega & \geq c_2^{-1} (1 + \alpha)^{-1} ((d_{k-1}^* d_{k-1})^{-1} q, q )_\Omega, 
\end{align}
with $c_2 :=c_{k - 1}^P+1$.

Setting $q=d_{k-1}^* v \in H \Lambda^{k-1}$, we note that $d_{k - 2}^*q = 0$ and therefore \eqref{eq:V_norms_bound_2} holds. By the Helmholtz decomposition of \Cref{lemma:Hodge}, we have $d_{k - 1}^*v = d_{k-1}^*d_{k-1} w$ and thus
\begin{align}\label{eq:V_norms_bound_3}
((\alpha I + &(1 + \alpha) d_{k-1}^* d_{k-1})^{-1} d_{k-1}^* v, d_{k-1}^* v)_\Omega \nonumber \\
& \ge c_2^{-1} ((1 + \alpha)^{-1} (d_{k-1}^* d_{k-1})^{-1} d_{k-1}^* v, d_{k-1}^* v)_\Omega  \nonumber \\
& = c_2^{-1} ((1 + \alpha)^{-1} (d_{k-1}^* d_{k-1})^{-1} d_{k-1}^*d_{k-1} w, d_{k-1}^*d_{k-1} w)_\Omega  \nonumber \\
& = c_2^{-1} (1 + \alpha)^{-1} (d_{k-1} w, d_{k-1} w)_\Omega \nonumber \\
& = c_2^{-1} (1 + \alpha)^{-1} \| d_{k-1} w \|_\Omega^2
\end{align}

We are left to bound the second component of the Hodge decomposition. For that, the Poincar\'e inequality (cf.~\Cref{lemma:Poincare}) gives us
\begin{align*}
\| d_{k}^* z \|_\Omega^2 & \le c_k^P \| d_{k} d_{k}^* z \|_\Omega^2
= c_k^P \| d_{k} v \|_\Omega^2 
\end{align*}
where we have also used~\eqref{eq:Helmholtz_1} and the fact that $d_k d_{k-1} w = 0$. Thus
\begin{align}\label{eq:V_norms_bound_4}
c_2^{-1} (1 + \alpha)^{-1} \| d_{k}^* z \|_\Omega^2 &\le c_2^{-1} (1 + \alpha)^{-1} c_k^P \| d_{k} v \|_\Omega^2 \nonumber \\
&\le c_2^{-1} c_k^P \| d_{k} v \|_\Omega^2,
\end{align}
because $\alpha > 0$. We note that the orthogonality of the Hodge decomposition gives
\begin{align}\label{eq:V_norms_bound_5}
	\| v \|_\Omega^2 = \| d_{k-1} w \|_\Omega^2 + \| d_{k}^* z \|_\Omega^2
\end{align}

Using \eqref{eq:V_norms_bound_3} and \eqref{eq:V_norms_bound_4} on the two components, respectively, we get
\begin{align}\label{eq:V_norms_bound_6}
c_2^{-1} (1 + \alpha)^{-1} \| v \|_\Omega^2 
& \le c_2^{-1} c_k^P \| d_{k} v \|_\Omega^2 
+ ((\alpha I + (1 + \alpha) d_{k-1}^* d_{k-1})^{-1} d_{k-1}^* v,d_{k-1}^* v)_\Omega
\end{align}

Multiplying by $c_2$ and adding $\| d_{k} v \|_\Omega^2$ to both sides of~\eqref{eq:V_norms_bound_6} allows us to conclude that
\begin{align*}
(1 + \alpha)^{-1} \| v \|_\Omega^2 + \| d_{k} v \|_\Omega^2
& \le (c_k^P +1) \| d_{k} v \|_\Omega^2 \nonumber\\
& \quad + c_2 (\alpha I + (1 + \alpha) d_{k-1}^* d_{k-1})^{-1} d_{k-1}^* v,d_{k-1}^* v)_\Omega \\
& \le c_1 \Vert v \Vert_V^2
\end{align*}
with $c_1 := \max \{ c_k^P + 1, c_2 \} = 1 + \max\{c_k^P, c_{k-1}^P\}$, completing the proof.
\end{proof}

In summary, from Theorem~\ref{thm:2}, Lemmata~\ref{lemma:0}--\ref{lemma:2} we conclude the following main result of this paper.

\begin{theorem}\label{thm:3}
The linear operator $\mathcal{P}: Y':=V' \times Q' \to V \times Q=:Y$ defined by
\begin{align}\label{eq: preconditioner}
	\mathcal{P} := \begin{bmatrix}
		(1 + \alpha)^{-1} I + d^* d & \\
		& \alpha I + (1 + \alpha) d^* d
	\end{bmatrix}^{-1}
\end{align}
provides a parameter-robust norm-equivalent preconditioner for the operator $\mathcal{A}$ associated with the Hodge-Laplacian in mixed form, i.e.,
$\kappa (\mathcal{P} \mathcal{A}) \le c$ where $c \ge 1$ is a constant that does not depend on the parameter $\alpha$.
\end{theorem}

\subsection{Reversing the definitions}
\label{sub:flipping_the_definitions}

In this section, we take a different perspective on the problem by choosing the fitted norms in the ``opposite direction''. Following \cite[Rem.~2.14]{hong2023new}, we first prescribe the full norm on $V$ and use that to define the norm on $Q$.
Let
\begin{align}
	|q|_Q^2 &\coloneqq \alpha \| q \|_\Omega^2, &
	|v|_V^2 &\coloneqq (1 + \alpha)^{-1} \| \Pi v \|_\Omega^2,
\end{align}
where $\Pi$ is the $L^2$ projection onto $d H \Lambda^{k - 1}$.
Let $|q|_{b}^2 = \langle B^T q, \bar V^{-1} B^T q \rangle$ so that the fitted norms (cf. \eqref{eqs:flipped_norms}) become
\begin{subequations}\label{eqs:flipped_VQ_norms}
\begin{align} 
	\| v \|_V^2 \label{eq:flipped_V_norm}
	&= |v|_V^2 + \langle A v, v \rangle
	= (1 + \alpha)^{-1} \| \Pi v \|_\Omega^2 + \| dv \|_\Omega^2 \\
	\| q \|_Q^2 
	&= |q|_Q^2 + |q|_{b}^2
	= \alpha \| q \|_\Omega^2 + \left((1 + \alpha)^{-1} \Pi + d^*d)^{-1} d q, d q \right)_\Omega.
\end{align}
\end{subequations}

Clearly, these norms are not practical in the context of preconditioning. We therefore first rewrite the $Q$-norm. Using the properties of the Hilbert complex and $\Pi$, we note that
\begin{align}
    ((1 + \alpha)^{-1} \Pi + d^*d) d q = (1 + \alpha)^{-1} \Pi dq + 0 = (1 + \alpha)^{-1} dq
\end{align}
and thus $((1 + \alpha)^{-1} \Pi + d^*d)^{-1} d q = (1 + \alpha) dq$. This allows us to rewrite the $Q$-norm as
\begin{align}
    \| q \|_Q^2 
    = \alpha \| q \|_\Omega^2 + (1 + \alpha) \| dq \|_\Omega^2,
\end{align}
which is exactly the same as \eqref{eq:Q_norm} in the previous subsection.

Following the same path as for the norms given by \eqref{eqs:VQ_norms}, we will verify that the norms \eqref{eqs:flipped_VQ_norms} satisfy the assumptions of Theorem~\ref{thm:2}.
Again, the continuity and coercivity of $c(\cdot , \cdot)$ in the $Q$-seminorm are obvious.
The fact that $| \cdot |_Q $ is a norm on ${\rm Ker}(B^T)$ trivially follows from its definition, because it is a norm on the whole space.
It remains to show the analogue of Lemma \ref{lemma:1}.
\begin{lemma} \label{lemma:flipped_b_infsup}
    The bilinear form $b(\cdot,\cdot)$ defining the linear operator $B^T$ is uniformly inf-sup stable in $\alpha$ for the fitted norms given by~\eqref{eqs:flipped_VQ_norms}, concretely
    \begin{align}
        \sup_{q \in Q} \frac{b(v, q)}{\norm{q}_Q} \geq \underline{\beta}^f |v|_V
    \end{align}
    for a constant $\underline{\beta}^f \geq  \left( c_{k-1}^P + 1\right)^{-\frac12}$ independent of $\alpha$.
\end{lemma} 
\begin{proof}
Let $v \in V$ be given. The $V$-norm only contains $\Pi v$, the projected component onto $d H \Lambda^{k - 1}$. From the Hodge decomposition (\Cref{lemma:Hodge}), we know that $\Pi v = d_{k - 1} q_0$ for some $q_0 \in H \Lambda^{k - 1}$ with $q_0 \perp \Ker d_{k-1}$.
We may therefore use the Poincaré inequality (\Cref{lemma:Poincare}) to obtain
\begin{align}
    \norm{q_0}_Q^2 & =  \alpha \norm{q_0}_\Omega^2 + (1 + \alpha) \norm{d_{k-1}q_0}_\Omega^2 \nonumber \\
                                & \leq \left( c_{k-1}^P \alpha + 1 + \alpha \right) \norm{d_{k-1}q_0}_\Omega^2 \nonumber \\
                                & \leq \left( c_{k-1}^P + 1\right) \left( 1 + \alpha\right) \norm{d_{k-1}q_0}_\Omega^2.
\end{align}

Next, we note that $(v, d_{k-1}q_0)_\Omega = (\Pi v, d_{k-1}q_0)_\Omega$ by the definition of the projection $\Pi$. Thus, choosing $q_0$ as the test function, we derive
\begin{align}
        \sup_{q \in Q} \frac{b(v, q)}{\norm{q}_Q} & \geq  \frac{b(v, q_0)}{\norm{q_0}_Q} = \frac{(\Pi v, d_{k-1} q_0)_\Omega}{\norm{q_0}_Q} \nonumber \\
        & \geq \left( c_{k-1}^P + 1\right)^{-\frac12} \left( 1 + \alpha\right)^{-\frac12} \frac{(\Pi v, d_{k-1}q_0)_\Omega}{ \norm{d_{k-1}q_0}_\Omega} \nonumber \\
        & = \left( c_{k-1}^P + 1\right)^{-\frac12} \left( 1 + \alpha\right)^{-\frac12} \norm{\Pi v}_\Omega = \left( c_{k-1}^P + 1\right)^{-\frac12} | v|_V.
\end{align}
\end{proof}

The following lemma shows that also the $V$-norm is equivalent to something more practical, leading to the preconditioner $\mathcal{P}$ given by \eqref{eq: preconditioner}.

\begin{lemma}\label{lemma:flipped_equivalence}
	The norm $ (1 + \alpha)^{-1} \norm{ v }_\Omega^2 + \norm{ dv }_\Omega^2$ is equivalent to the $V$-norm defined in~\eqref{eq:flipped_V_norm}, that is,
	\begin{align}\label{eq:flipped_equivalent_V_norms}
		\| v \|_V^2 \eqsim (1 + \alpha)^{-1} \| v \|_\Omega^2 + \| dv \|_\Omega^2, & \qquad \forall v \in V = H \Lambda^k.
	\end{align}
\end{lemma}

\begin{proof}
	The proof of both estimates uses the  Hodge decomposition $v = \Pi v + d_k^* z$.
	The $\lesssim$ direction trivially follows from the orthogonal decomposition since
	\begin{align}
		\norm{\Pi v}_\Omega^2 \leq \norm{\Pi v}_\Omega^2 + \norm{d_k^*z}_\Omega^2 = \norm{v}_\Omega^2 .
	\end{align}
	
	To prove the $\gtrsim$ direction we again use the Hodge decomposition to note that $d_kv = d_k d_k^* z$ and use the Poincaré inequality
	\begin{align}
		\norm{d_k v}_\Omega^2 &= \norm{d_k d_k^*z}_\Omega^2  \geq \left( c_k^P \right)^{-1} \norm{d_k^*z}_\Omega^2
    \end{align}
    
    However, we need to partially retain the norm on $d_k v$, so we apply this bound on half of the term:
    \begin{align}
		\norm{d_k v}_\Omega^2  & \geq  \left( 2c_k^P \right)^{-1} \norm{d_k^*z}_\Omega^2 + \frac12 \norm{d_k v}_\Omega^2, \nonumber \\
		   & \geq  \left( 2c_k^P \right)^{-1} \left( 1+\alpha \right)^{-1} \norm{d_k^*z}_\Omega^2 + \frac12 \norm{d_k v}_\Omega^2
    \end{align}
    From here, it straightforwardly follows that
    \begin{align}
    	 \norm{v}_V^2 & =  \left( 1+\alpha \right)^{-1} \norm{\Pi v}_\Omega^2 + \norm{d_k v}_\Omega^2 \nonumber \\
    	& \geq \left( 1+\alpha \right)^{-1} \norm{\Pi v}_\Omega^2 + \left( 2c_k^P \right)^{-1} \left( 1+\alpha \right)^{-1} \norm{d_k^*z}_\Omega^2 + \frac12 \norm{dv}_\Omega^2 \nonumber \\
    	& \geq \min\{1, \left( 2c_k^P \right)^{-1}\} \left( 1+\alpha \right)^{-1} \left( \norm{\Pi v}_\Omega^2 + \norm{d_k^*z}_\Omega^2  \right)+ \frac12 \norm{dv}_\Omega^2 \nonumber \\
    	& \geq \min\{1/2,\left(2c_k^P \right)^{-1} \}\left(   (1 + \alpha)^{-1} \norm{ v }_\Omega^2 + \norm{ d_k v }_\Omega^2 \right)
    \end{align}
    which completes the proof.
\end{proof}

In summary, the combination of \Cref{lemma:flipped_b_infsup,lemma:flipped_equivalence} and \Cref{thm:2} provide an alternative proof for \Cref{thm:3}.



\section{Numerical Results}
\label{sec:numerical_results}

Following the analysis of the previous section, we investigate the performance of the preconditioner~\eqref{eq: preconditioner}.
To maintain the estimates from the previous section, we need to consider discretization that preserves the structure that was used to derive them.
These conditions can be seen as a reformulation of the Galerkin property and the discrete $\inf\sup$ condition.
If we denote the discretizing space of $H\Lambda^k$ as $V_h^k$ then we need the spaces $V_h^k$ to satisfy the subcomplex property $d V_h^{k-1} \subset V_h^k$, $d V_h^{k} \subset V_h^{k+1}$ and the bounded projection property expressed by the diagram:
\begin{equation*} \label{eq:cochain_projection}
    \begin{tikzcd}
        H\Lambda^{k-1} \arrow[r, "d"] \arrow[d, "\pi_h^{k-1}"'] & H\Lambda^{k} \arrow[r, "d"] \arrow[d, "\pi_h^k"'] & H\Lambda^{k+1} \arrow[d, "\pi_h^{k+1}"'] \\
        V_h^{k-1} \arrow[r, "d"] & V_h^k \arrow[r, "d"] & V_h^{k+1}
    \end{tikzcd}
\end{equation*}
where the diagram commutes and the projection $\pi^k_h$ is bounded in the norm on $H\Lambda^k$ \cite{falk2014local}.
The discretizations that satisfy these properties and have sufficient aproximating properties are well known and studied, see \cite{arnold2006finite}. Importantly, the existence of such a projection ensures that the discrete sequence is itself an exact Hilbert complex and thus the results of the previous sections are directly inherited.

For the numerical experiments we use the Whitney forms that satisfy these conditions and justify the expectation of robustness with respect to $\alpha$ and mesh size. For $k = 0$, these correspond to the conventional, nodal Lagrange finite elements, whereas for $k = n$, they are the piecewise constants. The case $k = n - 1$ is the lowest-order facet-based Raviart-Thomas element \cite{raviart-thomas-0}. Finally, for $k = 1$ in 3D, the finite element is known as the edge-based Nédélec element of the first kind \cite{nedelec1-0}.

The experimental set-up is as follows.
We choose the domain to be $\Omega = [0, 1]^n$ with $n \in \{2, 3\}$, i.e. either the unit square or the unit cube. We then consider a sequence of five unstructured simplicial meshes $\{\Omega_h\}_h$ with decreasing $h$. We choose the parameter $\alpha \in \{ 10^{-4}, 10^{-2}, 1, 10^{2}, 10^{4} \}$ and set up the problem \eqref{eqs: Hodge Laplace problem} for $k \in \{ 1, \ldots, n \}$. The right-hand side is given by 
\begin{align}
	\langle g, q \rangle &= (\psi_{k - 1}, q)_\Omega, &
	\langle f, v \rangle &= (\psi_k, v)_\Omega
\end{align}
in which 
\begin{align}
	\psi_k &= 
	\begin{cases}
		\sum_{i = 1}^n \sin(2 \pi x_i), &k \in \{0, n\}, \\
		\sum_{i = 1}^n \sum_{j = 1}^n \sin(2 \pi x_i) \hat{\bm{e}}_j, & 1 \le k \le n - 1.
	\end{cases}
\end{align}

For each tuple $(h, \alpha, k)$, we apply the Minimal Residual (MinRes) method to the discretization of problem \eqref{eqs: Hodge Laplace problem} with the preconditioner \eqref{eq: preconditioner}. We then record the number of iterations required to reach an unpreconditioned relative residual of $10^{-7}$.

The numerical implementation uses the Python library PyGeoN \cite{pygeon} and the source code is freely available at \url{https://github.com/wmboon/preconditioning_HodgeLaplace}. All meshes are generated using GMSH \cite{geuzaine2009gmsh}. Moreover, the preconditioner $\mathcal{P}$ is implemented by employing direct solvers.

\subsection{Test case 1: the unit square}

In 2D, the Hodge Laplace problem with $k = 2$ corresponds to the mixed formulation of the Poisson equation, whereas the case $k = 1$ describes the vector Laplacian. 
We present the results for the two-dimensional test case in \Cref{tab: iterations 2D}. 

\begin{table}[ht]
\caption{Number of MinRes iterations for the Hodge Laplace problem \eqref{eqs: Hodge Laplace problem} with preconditioner \eqref{eq: preconditioner} on the unit square.}
\label{tab: iterations 2D}
\begin{tabular}{r|rrrrr|rrrrr}
\hline
& \multicolumn{5}{|c}{$k = 1$} & \multicolumn{5}{|c}{$k = 2$}\\
 $h \backslash \log_{10}(\alpha)$ & -4 & -2 & 0 & 2 & 4 & -4 & -2 & 0 & 2 & 4 \\
\hline
6.42e-02 & 4 & 4 & 5 & 5 & 4 & 2 & 2 & 4 & 5 & 4 \\
3.17e-02 & 4 & 4 & 5 & 5 & 4 & 2 & 2 & 4 & 5 & 4 \\
1.57e-02 & 4 & 4 & 4 & 5 & 4 & 2 & 2 & 4 & 5 & 4 \\
7.83e-03 & 4 & 4 & 4 & 4 & 4 & 2 & 2 & 4 & 5 & 4 \\
3.91e-03 & 4 & 4 & 4 & 4 & 4 & 2 & 2 & 4 & 5 & 5 
\end{tabular}
\end{table}

We clearly observe that the preconditioner is robust with respect to $\alpha$ and $h$, letting the Krylov subspace method converge within a maximum of five iterations. It is interesting to note that the preconditioner works particularly well for small values of $\alpha$ in the case of $k = 2$. This case corresponds to Darcy flow in highly permeable porous media. 

\subsection{Test case 2: the unit cube}

The three-dimensional setting gives rise to three distinct Hodge Laplace problems. The cases $k = 1$ and $k = 2$ correspond to different interpretations of the vector Laplacian, in which the dual variable is given by $p = \alpha^{-1} \nabla \cdot u$ or $p = \alpha^{-1} \nabla \times u$, respectively. We record our findings in \Cref{tab: iterations 3D}.

\begin{table}[ht]
\caption{Number of MinRes iterations for the Hodge Laplace problem \eqref{eqs: Hodge Laplace problem} with preconditioner \eqref{eq: preconditioner} on the unit cube.}
\label{tab: iterations 3D}
\begin{tabular}{r|rrrrr|rrrrr|rrrrr}
\hline
& \multicolumn{5}{|c}{$k = 1$} & \multicolumn{5}{|c}{$k = 2$} & \multicolumn{5}{|c}{$k = 3$}\\
 $h \backslash \log_{10}(\alpha)$ & -4 & -2 & 0 & 2 & 4 & -4 & -2 & 0 & 2 & 4 & -4 & -2 & 0 & 2 & 4 \\
\hline
3.79e-01 & 3 & 4 & 5 & 5 & 4 & 2 & 3 & 4 & 5 & 4 & 2 & 2 & 4 & 5 & 4 \\
2.62e-01 & 3 & 4 & 5 & 4 & 3 & 2 & 3 & 4 & 4 & 4 & 2 & 2 & 4 & 5 & 4 \\
1.99e-01 & 3 & 4 & 5 & 4 & 3 & 2 & 3 & 4 & 4 & 4 & 2 & 2 & 4 & 5 & 4 \\
1.37e-01 & 3 & 4 & 5 & 4 & 3 & 2 & 3 & 4 & 4 & 4 & 2 & 2 & 4 & 5 & 4 \\
9.06e-02 & 3 & 3 & 4 & 4 & 3 & 1 & 3 & 4 & 4 & 4 & 2 & 3 & 4 & 5 & 4 
\end{tabular}
\end{table}

Again, we observe that the preconditioner is robust with respect to $\alpha$ and $h$, for all three instances of $k$. Similar to the two-dimensional case, the preconditioner performs very well for the Darcy case $k = 3$ with small parameter $\alpha$, which corresponds physically to a high permeability. A similar behavior is observed for the case $k = 2$.

\section{Conclusion}
\label{sec:conclusion}

We have analyzed the Hodge Laplacian in mixed form in the abstract framework for perturbed saddle point problems recently proposed in~\cite{hong2023new}. This framework provides a norm fitting technique to construct parameter-dependent norms in which the problem is uniformly well-posed. In some instances, as demonstrated in this work for the Hodge-Laplacian, these norms can be further simplified for the purpose of making the induced norm-equivalent block diagonal preconditioners easier to implement.

We then analyzed the system from a different perspective and showed that this leads to the same weighted norms and preconditioner. The performance of this preconditioner was tested numerically on the Hodge Laplace problems on the discrete de Rham complex using low-order finite elements in 2D and 3D. The experiments showed that the preconditioner is robust with respect to the mesh size $h$ and the weighting parameter $\alpha$. 

By presenting the analysis in a general setting, our results are directly applicable to any Hodge Laplace problem on an exact Hilbert complex. The proposed preconditioning strategy is therefore also robust for higher-order finite elements in the Finite Element Exterior Calculus \cite{arnold2006finite} framework, discrete de Rham methods \cite{di2023arbitrary}, and Virtual Element Methods \cite{beirao2013basic}. Moreover, the approach is not limited to the problems on the de Rham complex and applies, among others, to the Hodge Laplace problems on the elasticity, Stokes, and Cosserat complexes.

{\small
\printbibliography
}
	
{\small
{\em Authors' addresses}:
{\em Wietse M. Boon}, Division of Energy and Technology, NORCE Norwegian Research Centre, Bergen, Norway,
 e-mail: \texttt{wibo@\allowbreak norceresearch.no};
{\em Johannes Kraus, Maria Lymbery}, Faculty of Mathematics, University of Duisburg-Essen, Thea-Leymann-Str.~9,
D-45127 Essen, Germany, e-mail: \texttt{johannes.kraus@\allowbreak uni-due.de, maria.lymbery@\allowbreak uni-due.de};
{\em Tomáš Luber}, Institute of Geonics, Czech Academy of Sciences, Studenstká 1768, 70800 Ostrava, Czech Republic,
e-mail: \texttt{ tomas.luber@\allowbreak ugn.cas.cz}.
}

\end{document}